\documentclass[11pt,reqno]{amsart}

\usepackage{amsmath,amssymb,mathrsfs}
\usepackage{graphicx,cite,times}

\newtheorem{theo}{Theorem}[section]

\newtheorem{lemm}[theo]{Lemma}

\newtheorem{rema}[theo]{Remark}

\numberwithin{equation}{section}

\begin{document}

\title[stability on an random source problem]{Stability on the Inverse Random
Source Scattering Problem for the One-Dimensional Helmholtz Equation}

\author{Peijun Li}
\address{Department of Mathematics, Purdue University, West Lafayette, Indiana
47907, USA.}
\email{lipeijun@math.purdue.edu}

\author{Ganghua Yuan}
\address{KLAS, School of Mathematics and Statistics, Northeast Normal University,
Changchun, Jilin, 130024, China}
\email{yuangh925@nenu.edu.cn}

\thanks{The research of PL was supported in part by the NSF grant DMS-1151308.
The research of GY was supported in part by NSFC grants 10801030, 11271065,
11571064, the Ying Dong Fok Education Foundation under grant 141001, and the
Fundamental Research Funds for the Central Universities under grant
2412015BJ011.}


\begin{abstract}
Consider the one-dimensional stochastic Helmholtz equation where the source is
assumed to be driven by the white noise. This paper concerns the stability
analysis of the inverse random source problem which is to reconstruct the
statistical properties of the source such as the mean and variance. Our results
show that increasing stability can be obtained for the inverse problem by using
suitable boundary data with multi-frequencies.
\end{abstract}

\maketitle

\section{Introduction and problem statement}

Consider the one-dimensional stochastic Helmholtz equation
\begin{equation}\label{she}
 u''(x, \kappa)+\kappa^2 u(x, \kappa)=f(x) +\sigma(x)\dot{W}_x,
\end{equation}
where $\kappa>0$ is the wavenumber, $f$ and $\sigma$ are deterministic
functions which have compact supports contained in the interval $[0,\,1]$, $W_x$
is the spatial Brownian motion and $\dot{W}_x$ is the white noise. In this
model, $f, \sigma$, and $\sigma^2$ can be viewed as the mean, the standard
deviation, and the variance of the random source, respectively. The radiated
random wave field $u$ is required to satisfy the outgoing wave conditions:
\begin{equation}\label{owc}
 u'(0, \kappa)+{\rm i}\kappa u(0, \kappa)=0,\quad u'(1, \kappa)-{\rm i}\kappa
u(1, \kappa)=0.
\end{equation}

Given $f$ and $\sigma$, the direct source scattering problem is to determine the
radiated wave field $u$. It is shown in \cite{BCLZ-MC14} that
\eqref{she}--\eqref{owc} has a unique pathwise solution which is
\begin{equation}\label{sol}
 u(x, \kappa)=\int_0^1 \frac{e^{{\rm i}\kappa |x-y|}}{2{\rm i}\kappa}f(y){\rm
d}y+\int_0^1  \frac{e^{{\rm i}\kappa |x-y|}}{2{\rm i}\kappa}\sigma(y){\rm d}W_y.
\end{equation}
Here the second integral at the right hand side of \eqref{sol} is understood in
the sense of It\^{o}. This paper concerns the inverse source scattering
problem, which is to determine $f$ and $g=\sigma^2$ from boundary measurement of
the radiated wave field $u$. Specifically, we propose the following two inverse
problems:

\begin{enumerate}

\item If $f$ and $g$ are complex function, the inverse problem is to determine
$f$ and $g$ simultaneously by two-sided observation data $u(0, \kappa)$
and $u(1, \kappa), \kappa\in(0, K)$, where $K>1$ is a constant.

\item If $f$ is a real function, the inverse problem is to determine $f$
by one-sided observation data $u(0, \kappa), \kappa\in (0, 1)\cup\cup_{j=1}^N
j\pi $, where $N\in\mathbb{N}$.

\end{enumerate}

The inverse source problem has significant applications in medical and
biomedical imaging \cite{I-89}. Although the deterministic inverse source
problem has been well studied \cite{BLLT-IP15, BLRX-SJNA15}, little is known for
the stochastic case \cite{D-JMP79}. We refer to \cite{BCLZ-MC14, L-IP11} for
numerical solution of the one-dimensional inverse random source scattering
problem. A related inverse random source problem can be found in \cite{BX-IP13}.
However, there are no stability results available for the inverse random source
scattering problem at present.

In this paper, we study stability of the above two inverse problems. As is
known, the inverse source problem does not have a unique solution at a single
frequency even for its deterministic counterpart
\cite{DS-IEEE82, HKP-IP05}. Our goal is to establish increasing stability of the
inverse problems with multi-frequencies. We refer to \cite{BLT-JDE10, CIL} for
increasing stability of the deterministic inverse source problem. In \cite{CIL},
the authors discussed stability of the inverse source problem for the
three-dimensional Helmholtz equation by using the Huygens principle. In
\cite{BLT-JDE10}, the authors studied the stability of the two- and
three-dimensional Helmholtz equations via Green's functions. Related results
can be found in \cite{I-CM07, I-D11} on increasing stability of determining
potentials and in the continuation for the Helmholtz equation.

\section{Main results}

Let the triple $(\Omega, \mathcal{F}, P)$ be a complete probability space on
which the one-dimensional Brownian motion $\{W_x\}_{x\in [0, 1]}$ is defined. If
$X$ is a random variable, ${\bf E}(X)$ and ${\bf V}(X)={\bf
E}(X-{\bf E}(X))^2$ are the expectation and variance of $X$, respectively. We
remark that ${\bf V}(X)$ is not an ordinary variance if $X$ is a complex-valued
random variable. For convenience, we still call ${\bf V}(X)$ the variance of
random variable $X$ even if it is complex-valued. We refer to \cite{F-06} for
more details on notation of stochastic differential equations.

Define a complex-valued functional space:
\[
 \mathcal{C}_M =\{f\in H^n(0, 1): \|f\|_{H^n(0, 1)}\leq M, ~{\rm
supp}f\subseteq(0, 1), ~f: (0, 1)\to\mathbb{C}\}
\]
and a real-valued functional space:
\[
 \mathcal{R}_M =\{f\in H^n(0, 1): \|f\|_{H^n(0, 1)}\leq M, ~{\rm
supp}f\subseteq(0, 1), ~f: (0, 1)\to\mathbb{R}\},
\]
where $n\in\mathbb{N}$ and $M>1$ is a constant. Given two random functions $u_1$
and $u_2$, we define the function of expectation discrepancy:
\[
 v(x, \kappa)={\bf E}u_1(x, \kappa)-{\bf E}u_2(x, \kappa)
\]
and the function of the variance discrepancy:
\[
 w(x, \kappa)={\bf V}u_1(x, \kappa)-{\bf V}u_2(x, \kappa).
\]

Now we show the main stability result of the first inverse problem.

\begin{theo}\label{mr1}
 Let $f_j, g_j\in\mathcal{C}_M, j=1, 2$, and let $u_j$ be the solution
\eqref{sol} corresponding to $f_j, g_j$. Then there exist two
positive constants $C_1, C_2$ independent of $n, K, M, \kappa$ such that
\begin{align}
\label{cfe} \| f_1-f_2\|^2_{L^2(0, 1)}\leq C_1
\left(\epsilon^2_1+\frac{M^2}{\left(\frac{K^{\frac{2}{3}}|\ln
\epsilon_1|^{\frac{1}{4}}}{(6n-3)^3}\right)^{ 2n-1}} \right),\\
\label{cge}\| g_1-g_2\|^2_{L^2(0, 1)}\leq C_2
\left(\epsilon^2_2+\frac{M^2}{\left(\frac{K^{\frac{2}{3}}|\ln
\epsilon_2|^{\frac{1}{4}}}{(6n-3)^3}\right)^{ 2n-1}}\right),
\end{align}
where $K>1$ and
\begin{align}
\label{e1} \epsilon_1&=\left(4\int_0^K \kappa^2\left( |v(0, \kappa)|^2 +
|v(1, \kappa)|^2 \right){\rm d}\kappa\right)^{\frac{1}{2}},\\
\label{e2}\epsilon_2&=\left(16\int_0^K \kappa^4\left( |w(0, \kappa)|^2 +
|w(1, \kappa)|^2 \right){\rm d}\kappa\right)^{\frac{1}{2}}.
\end{align}
\end{theo}

\begin{rema}
There are two parts in the stability estimates \eqref{cfe} and \eqref{cge}: the
first part is the data discrepancy and the second part comes from the high
frequency tails of the functions. It is clear to see that the stability
increases as $K$ increases, i.e., the problem is more stable as more
frequencies data are used. We can also see that when
$n<\left[\frac{K^{\frac{2}{9}}|\ln\epsilon_j|^{\frac{1}{12}}+3}{6}\right]$ the the
stability increases as $n$ increases, i.e., the problem is more
stable as the functions have suitably higher regularity.
\end{rema}

Here is the main stability result of the second inverse problem.
\begin{theo}\label{mr2}
 Let $f_j\in\mathcal{R}_M, j=1, 2,$ and let $u_j$ be the solution of
\eqref{sol} corresponding to $f_j$. Let
\[
\epsilon_3=\left(\sum_{j=1}^N (2j\pi)^2
|{\rm Re}v(0, j\pi)|^2\right)^{\frac{1}{2}},\quad
\epsilon_4=\sup_{\kappa\in (0, 1)}2\kappa|{\rm Re}v(0, \kappa)|<1.
\]
Then there exists a positive constant $C_3$ independent of $n,
N, M, \kappa$ such that
\[
 \|f_1-f_2\|^2_{L^2(0, 1)}\leq C_3 \left(\epsilon^2_3+\frac{M^2}{\left(\frac{
N^{\frac{5}{8}}|\ln\epsilon_4|^{\frac{1}{9}}}{(6n-3)^3} \right)^{2n-1}} \right).
\]
\end{theo}

\begin{rema}
Theorem \ref{mr2} shows that only one-sided boundary observation data are
needed for the wavenumbers in the set $(0, 1)\cup\cup_{j=1}^N j\pi$ if one wants
to determine the mean of the random source. The stability increases as $N$ or
$n<\left[\frac{N^{\frac{5}{24}}|\ln\epsilon_4|^{\frac{1}{27}}+3}{6}\right]$
increases.
\end{rema}

The remainder of the paper is organized as follows. We prove Theorem \ref{mr1}
and Theorem \ref{mr2} in section 3 and section 4, respectively.

\section{Proof of Theorem \ref{mr1}}

First we present several useful lemmas.

\begin{lemm}
 Let $f_j, g_j\in L^2(0, 1), j =1, 2$. We have
 \begin{align*}
 \|f_1-f_2\|^2_{L^2(0, 1)}&=\frac{2}{\pi}\int_0^\infty \kappa^2 \Bigl(|v(0,
\kappa)|^2+|v(1, \kappa)|^2 \Bigr) {\rm d}\kappa,\\
  \|g_1-g_2\|^2_{L^2(0, 1)}&=\frac{16}{\pi}\int_0^\infty \kappa^4 \Bigl(
|w(0, \kappa)|^2+|w(1, \kappa)|^2\Bigr) {\rm d}\kappa.
 \end{align*}
\end{lemm}

\begin{proof}
 Letting $\xi\in\mathbb{R}$ with $|\xi|=\kappa$, we multiply $e^{-{\rm i}\xi
x}$ on both sides of \eqref{she} and obtain
\[
 e^{-{\rm i}\xi x}u''(x, \kappa)+\kappa^2 e^{-{\rm i}\xi x} u(x,
\kappa)=e^{-{\rm i}\xi x} f(x)+e^{-{\rm i}\xi x} \sigma(x) \dot{W}_x.
\]
Since
\[
(e^{-{\rm i}\xi x}u'(x, \kappa))'=e^{-{\rm i}\xi x} u''(x, \kappa)-{\rm i}\xi
e^{-{\rm i}\xi x} u'(x, \kappa),
\]
we have
\begin{equation}\label{a1}
 (e^{-{\rm i}\xi x}u'(x, \kappa))'=e^{-{\rm i}\xi x} f(x)+e^{-{\rm i}\xi
x}\sigma(x)\dot{W}_x-\kappa^2 e^{-{\rm i}\xi x} u(x, \kappa)-{\rm i}\xi
e^{-{\rm i}\xi x} u'(x, \kappa).
\end{equation}
Integrating \eqref{a1} over $(0, 1)$ with respect to $x$ yields
\begin{align}\label{a2}
 e^{-{\rm i}\xi}u'(1, \kappa)-u'(0, \kappa)=&\int_0^1 e^{-{\rm i}\xi x} f(x){\rm
d}x+\int_0^1 e^{-{\rm i}\xi x}\sigma(x){\rm d}W_x\notag\\
&-\kappa^2\int_0^1 e^{-{\rm i}\xi x} u(x, \kappa){\rm d}x-{\rm i}\xi
\int_0^1 e^{-{\rm i}\xi x} u'(x, \kappa){\rm d}x.
\end{align}
It follows from the integration by parts that
\begin{equation}\label{a3}
 -{\rm i}\xi \int_0^1 e^{-{\rm i}\xi x} u'(x, \kappa){\rm d}x=-{\rm i}\xi
e^{-{\rm i}\xi}u(1, \kappa)+{\rm i}\xi u(0, \kappa)+\kappa^2\int_0^1 e^{-{\rm
i}\xi x}u(x, \kappa){\rm d}x.
\end{equation}
Substituting \eqref{a3} into \eqref{a2}, we get
\[
 e^{-{\rm i}\xi}u'(1, \kappa)+{\rm i}\xi e^{-{\rm i}\xi} u(1, \kappa)-u'(0,
\kappa)-{\rm i}\xi u(0, \kappa)=\int_0^1 e^{-{\rm i}\xi x}f(x){\rm d}x
+\int_0^1 e^{-{\rm i}\xi x} \sigma(x){\rm d}W_x,
\]
which gives after applying the outgoing wave conditions \eqref{owc} that
\begin{equation}\label{a4}
 {\rm i}(\kappa+\xi)e^{-{\rm i}\xi}u(1, \kappa)+{\rm i}(\kappa-\xi)u(0,
\kappa)=\int_0^1 e^{-{\rm i}\xi x}f(x){\rm d}x +\int_0^1 e^{-{\rm i}\xi x}
\sigma(x){\rm d}W_x.
\end{equation}
Taking the expectation on both sides of \eqref{a4}, we obtain
\begin{equation}\label{a5}
 \int_0^1 e^{-{\rm i}\xi x}f(x){\rm d}x={\rm i}(\kappa+\xi)e^{-{\rm i}\xi}{\bf
E}u(1, \kappa)+{\rm i}(\kappa-\xi){\bf E}u(0, \kappa), \quad|\xi|=\kappa\in (0,
\infty).
\end{equation}
Since $f_j$ is assumed to have a compact support in $[0, 1]$, we have from
\eqref{a5} that
\[
 \hat{f}_j(\xi)= \int_{-\infty}^\infty e^{-{\rm i}\xi x}f_j(x){\rm d}x={\rm
i}(\kappa+\xi)e^{-{\rm i}\xi}{\bf E}u_j(1, \kappa)+{\rm i}(\kappa-\xi){\bf
E}u_j(0, \kappa), \quad|\xi|=\kappa\in (0, \infty),
\]
which gives
\begin{align*}
\hat{f}_1(\xi)-\hat{f}_2(\xi) =\int_{-\infty}^\infty e^{-{\rm i}\xi
x}(f_1(x)-f_2(x)){\rm d}x={\rm i}(\kappa+\xi)e^{-{\rm i}\xi}({\bf E}u_1(1,
\kappa)-{\bf E}u_2(1, \kappa))\\
+{\rm i}(\kappa-\xi)({\bf E}u_1(0, \kappa)-{\bf E}u_2(0, \kappa)),
\quad|\xi|=\kappa\in (0, \infty).
\end{align*}
Hence we have
\[
 \hat{f}_1(-\kappa)-\hat{f}_2(-\kappa)=2{\rm i}\kappa ({\bf E}u_1(0,
\kappa)-{\bf E}u_2(0, \kappa))=2{\rm i}\kappa v(0, \kappa)
\]
and
\[
\hat{f}_1(\kappa)-\hat{f}_2(\kappa)=2{\rm i}e^{-{\rm i}\kappa}\kappa({\bf
E}u_1(1, \kappa)-{\bf E}u_2(1, \kappa))=2{\rm i}\kappa e^{-{\rm i}\kappa} v(1,
\kappa).
\]

It follows from the Plancherel theorem that
\begin{align*}
 \|f_1-f_2\|^2_{L^2(0, 1)}&=\frac{1}{2\pi}\int_{-\infty}^\infty|\hat{f}
_1(\xi)-\hat{f}_2(\xi)|^2{\rm d}\xi\\
&=\frac{1}{2\pi}\int_{-\infty}^0 |\hat{f}
_1(\xi)-\hat{f}_2(\xi)|^2{\rm d}\xi+\frac{1}{2\pi}\int_0^\infty |\hat{f}
_1(\xi)-\hat{f}_2(\xi)|^2{\rm d}\xi\\
&=\frac{1}{2\pi}\int_0^\infty |\hat{f}_1(-\kappa)-\hat{f}_2(-\kappa)|^2{\rm
d}\kappa+\frac{1}{2\pi}\int_0^\infty |\hat{f}_1(\kappa)-\hat{f}_2(\kappa)|^2{
\rm d}\kappa\\
&=\frac{2}{\pi}\int_0^\infty \kappa^2|v(0, \kappa)|^2{\rm d}\kappa +
\frac{2}{\pi}\int_0^\infty \kappa^2|v(1, \kappa)|^2{\rm d}\kappa.
\end{align*}
Noting
\[
 {\bf E}\Bigl(\int_0^1 e^{-{\rm i}\xi x}\sigma(x){\rm d}W_x\Bigr)^2=\int_0^1
e^{-2{\rm i}\xi x}\sigma^2(x){\rm d}x=\int_0^1 e^{-2{\rm i}\xi x}g(x){\rm
d}x,
\]
we have from \eqref{a4} that
\begin{equation}\label{a6}
 \int_0^1 e^{-2{\rm i}\xi x}g(x){\rm d}x={\bf V}\bigl({\rm
i}(\kappa+\xi)e^{-{\rm i}\xi}u(1, \kappa)+{\rm i}(\kappa-\xi)u(0,
\kappa)\bigr), \quad|\xi|=\kappa\in(0, \infty).
\end{equation}
Since $g_j$ has a compact support in $(0, 1)$, we get from \eqref{a6} that
\[
\hat{g}_j(2\xi)= \int_{-\infty}^\infty e^{-2{\rm i}\xi x}g_j(x){\rm d}x={\bf
V}\bigl({\rm i}(\kappa+\xi)e^{-{\rm i}\xi}u_j(1, \kappa)+{\rm
i}(\kappa-\xi)u_j(0, \kappa)\bigr), \quad|\xi|=\kappa\in(0, \infty),
\]
which gives
\begin{align*}
& \hat{g}_1(2\xi)-\hat{g}_2(2\xi)= \int_{-\infty}^\infty e^{-2{\rm i}\xi
x}(g_1(x)-g_2(x)){\rm d}x\\
&={\bf V}\bigl({\rm i}(\kappa+\xi)e^{-{\rm i}\xi}u_1(1, \kappa)+{\rm
i}(\kappa-\xi)u_1(0, \kappa)\bigr)-{\bf
V}\bigl({\rm i}(\kappa+\xi)e^{-{\rm i}\xi}u_2(1, \kappa)+{\rm
i}(\kappa-\xi)u_2(0, \kappa)\bigr).
\end{align*}
Hence we have
\[
 \hat{g}_1(-2\kappa)-\hat{g}_2(-2\kappa)=(2{\rm i}\kappa)^2 \bigl({\bf
V}u_1(0, \kappa)-{\bf V}u_2(0, \kappa)\bigr)=(2{\rm i}\kappa)^2 w(0, \kappa)
\]
and
\[
 \hat{g}_1(2\kappa)-\hat{g}_2(2\kappa)=(2{\rm i}\kappa)^2 e^{-2{\rm i}\kappa}
\bigl({\bf V}u_1(1, \kappa)-{\bf V}u_2(1, \kappa)\bigr)=(2{\rm i}\kappa)^2
e^{-2{\rm i}\kappa} w(1, \kappa).
\]
Using the Plancherel theorem again yields
\begin{align*}
 \|g_1-g_2\|^2_{L^2(0, 1)}&=\frac{1}{2\pi}\int_{-\infty}^\infty|\hat{g}
_1(\xi)-\hat{g}_2(\xi)|^2{\rm d}\xi=\frac{1}{\pi}\int_{-\infty}^\infty|\hat{g}
_1(2\xi)-\hat{g}_2(2\xi)|^2{\rm d}\xi\\
&=\frac{1}{\pi}\int_0^\infty |\hat{g}
_1(-2\kappa)-\hat{g}_2(-2\kappa)|^2{\rm d}\kappa+\frac{1}{\pi}\int_0^\infty
|\hat{g} _1(2\kappa)-\hat{g}_2(2\kappa)|^2{\rm d}\kappa\\
&=\frac{16}{\pi}\int_0^\infty \kappa^4 |w(0, \kappa)|^2{\rm
d}\kappa+\frac{16}{\pi}\int_0^\infty \kappa^2 |w(1, \kappa)|^4{\rm d}\kappa,
\end{align*}
which completes the proof.
\end{proof}

\begin{lemm}
 Let $f_j, g_j\in L^2(0, 1), j =1, 2$. We have
 \begin{align*}
  4\kappa^2 |v(0, \kappa)|^2&=\Bigl|\int_0^1
e^{{\rm i}\kappa x}(f_1(x)-f_2(x)){\rm d}x\Bigr|^2,\\
4\kappa^2 |v(1, \kappa)|^2&=\Bigl|\int_0^1
e^{-{\rm i}\kappa x}(f_1(x)-f_2(x)){\rm d}x\Bigr|^2,\\
16\kappa^4 |w(0, \kappa)|^2&=\Bigl|\int_0^1
e^{2{\rm i}\kappa x}(g_1(x)-g_2(x)) {\rm d}x\Bigr|^2,\\
16\kappa^4 |w(1, \kappa)|^2&=\Bigl|\int_0^1
e^{-2{\rm i}\kappa x}(g_1(x)-g_2(x)) {\rm d}x\Bigr|^2.
 \end{align*}
\end{lemm}

\begin{proof}
 It follows from \eqref{sol} that the solution of \eqref{she}--\eqref{owc}
is
\[
2{\rm i}\kappa u_j(x, \kappa)=\int_0^1 e^{{\rm i}\kappa|x-y|}f_j(y){\rm
d}y+\int_0^1 e^{{\rm i}\kappa|x-y|}\sigma_j(y){\rm d}W_y,
\]
which gives
\begin{align}
\label{b1} 2{\rm i}\kappa u_j(0, \kappa)&=\int_0^1 e^{{\rm i}\kappa x}f_j(x){\rm
d}x+\int_0^1 e^{{\rm i}\kappa x}\sigma_j(x){\rm d}W_x,\\
\label{b2} 2{\rm i}\kappa u_j(1, \kappa)&=\int_0^1 e^{{\rm i}\kappa
(1-x)}f_j(x){\rm
d}x+\int_0^1 e^{{\rm i}\kappa (1-x)}\sigma_j(x){\rm d}W_x.
\end{align}
Taking expectation of \eqref{b1} and \eqref{b2}, we may obtain
\begin{align*}
 2{\rm i}\kappa v(0, \kappa)&=\int_0^1 e^{{\rm
i}\kappa x}(f_1(x)-f_2(x)){\rm d}x,\\
 2{\rm i}\kappa v(1, \kappa)&=\int_0^1  e^{{\rm
i}\kappa (1-x)}(f_1(x)-f_2(x)){\rm d}x.
\end{align*}
Taking the variance of \eqref{b1} and \eqref{b2} yields
\begin{align*}
 -4\kappa^2 w(0, \kappa)&=\int_0^1 e^{2{\rm
i}\kappa x}(g_1(x)-g_2(x)){\rm d}x,\\
-4\kappa^2 w(1, \kappa)&=\int_0^1  e^{2{\rm
i}\kappa(1- x)}(g_1(x)-g_2(x)){\rm d}x,
\end{align*}
which completes the proof by taking square of the amplitudes on both sides of
the above four equations.
\end{proof}

Let
\begin{align}
\label{i1}I_1(s)&=\int_0^s \left|\int_0^1 e^{{\rm i}\kappa x}(f_1(x)-f_2(x)){\rm
d}x\right|^2{\rm d}\kappa+\int_0^s \left|\int_0^1 e^{-{\rm i}\kappa
x}(f_1(x)-f_2(x)){\rm d}x\right|^2{\rm d}\kappa,\\
\label{i2}I_2(s)&=\int_0^s \left|\int_0^1 e^{2{\rm i}\kappa
x}(g_1(x)-g_2(x)){\rm d}x \right|^2{\rm d}\kappa+\int_0^s \left|\int_0^1
e^{-2{\rm i}\kappa x}(g_1(x)-g_2(x)){\rm d}x\right|^2{\rm d}\kappa.
\end{align}
The integrands in \eqref{i1} and \eqref{i2} are entire analytic function of
$\kappa$. The integrals with respect to $s$ can be taken over any path joining
points $0$ and $\kappa$ in the complex plane. Thus $I_1(s)$ and $I_2(s)$ are
entire analytic functions of $s=s_1+{\rm i}s_2, s_1, s_2\in\mathbb{R}$.

\begin{lemm}
 Let $f_j, g_j\in L^2(0, 1), j =1, 2$. We have  for any $s=s_1+{\rm i}s_2, s_1,
s_2\in\mathbb{R}$ that
 \begin{align*}
  |I_1(s)|&\leq 2|s|e^{2|s_2|}\int_0^1 |f_1(x)-f_2(x)|^2{\rm d}x,\\
  |I_2(s)|&\leq 2|s|e^{4|s_2|}\int_0^1 |g_1(x)-g_2(x)|^2{\rm d}x.
 \end{align*}
\end{lemm}

\begin{proof}
 Let $\kappa=st, t\in(0, 1)$. A simple calculation yields
 \begin{align*}
  I_1(s)=s\int_0^1 \left|\int_0^1 e^{{\rm i}st x}(f_1(x)-f_2(x)){\rm d}x
\right|^2{\rm d}t+s\int_0^1  \left|\int_0^1 e^{-{\rm i}st x}(f_1(x)-f_2(x)){\rm
d}x \right|^2{\rm d}t.
\end{align*}
Noting that $|e^{\pm{\rm i}st x}|\leq e^{|s_2|}$ for all $x\in(0, 1)$,
we have
\[
 |I_1(s)|\leq 2|s|\int_0^1 \Bigl(\int_0^1 e^{2|s_2|}|f_1(x)-f_2(x)|^2{\rm d}x
\Bigr){\rm d}t\leq 2|s|e^{2|s_2|}\int_0^1 |f_1(x)-f_2(x)|^2 {\rm d}x.
\]
Similarly, we can show that
\[
  |I_2(s)|\leq 2|s|e^{4|s_2|}\int_0^1 |g_1(x)-g_2(x)|^2{\rm d}x.
\]
which completes the proof.
\end{proof}

\begin{lemm}
  Let $f_j, g_j\in H^n(0, 1), j =1, 2$. We have for any $s>0$ that
\begin{align*}
 4\int_s^\infty \kappa^2 \bigl(|v(0, \kappa)|^2 +|v(1, \kappa)|^2 \bigr){\rm
d}\kappa&\leq 2 s^{-(2n-1)}\| f_1-f_2\|^2_{H^n(0, 1)},\\
 16\int_s^\infty \kappa^4 \bigl(|w(0, \kappa)|^2 +|w(1, \kappa)|^2 \bigr){\rm
d}\kappa&\leq 2 s^{-(2n-1)}\| g_1-g_2\|^2_{H^n(0, 1)}.
\end{align*}
\end{lemm}

\begin{proof}
It follows from Lemma 3.2 that we have
 \begin{align*}
  4&\int_s^\infty \kappa^2 |v(0, \kappa)|^2{\rm
d}\kappa+ 4\int_s^\infty \kappa^2 |v(1, \kappa)|^2{\rm d}\kappa\\
&=\int_s^\infty \Bigl|\int_0^1 e^{{\rm i}\kappa x}(f_1(x)-f_2(x)){\rm
d}x\Bigr|^2 {\rm d}\kappa +\int_s^\infty \Bigl|\int_0^1 e^{-{\rm i}\kappa
x}(f_1(x)-f_2(x)){\rm d}x\Bigr|^2
{\rm d}\kappa
 \end{align*}
 Using integration by parts and noting ${\rm supp}f_j\in (0, 1)$, we obtain
 \[
  \int_0^1 e^{\pm{\rm i}\kappa x}(f_1(x)-f_2(x)){\rm d}x=\frac{1}{(\pm{\rm
i}\kappa)^n}\int_0^1 e^{\pm{\rm i}\kappa x}(f_1^{(n)}(x)-f_2^{(n)}(x)){\rm d}x,
 \]
which gives
\[
 \Bigl|\int_0^1 e^{\pm{\rm i}\kappa x}(f_1(x)-f_2(x)){\rm d}x\Bigr|^2\leq
\kappa^{-2n}\|f^{(n)}_1-f_2^{(n)}\|^2_{H^n(0, 1)}
\]
Hence we get
\begin{align*}
 \int_s^\infty \Bigl|\int_0^1 e^{\pm{\rm i}\kappa x}(f_1(x)-f_2(x)){\rm
d}x\Bigr|^2 {\rm d}\kappa&\leq \|f^{(n)}_1-f_2^{(n)}\|^2_{H^n(0, 1)}
\int_s^\infty \kappa^{-2n}{\rm d}\kappa\\
&=\frac{s^{-(2n-1)}}{(2n-1)}\|f^{(n)}_1-f_2^{(n)}\|^2_{H^n(0, 1)}
\end{align*}

Again, we have from Lemma 3.2 that
\begin{align*}
16&\int_s^\infty \kappa^4 |w(0, \kappa)|^2{\rm
d}\kappa+ 16\int_s^\infty \kappa^4 |w(1, \kappa)|^2{\rm d}\kappa\\
&=\int_s^\infty \Bigl|\int_0^1 e^{2{\rm i}\kappa x}(g_1(x)-g_2(x)){\rm
d}x\Bigr|^2 {\rm d}\kappa +\int_s^\infty \Bigl|\int_0^1 e^{-2{\rm i}\kappa
x}(g_1(x)-g_2(x)){\rm d}x\Bigr|^2 {\rm d}\kappa.
\end{align*}
Similarly, we have
 \[
  \int_0^1 e^{\pm 2{\rm i}\kappa x}(g_1(x)-g_2(x)){\rm d}x=\frac{1}{(\pm 2{\rm
i}\kappa)^n}\int_0^1 e^{\pm 2{\rm i}\kappa x}(g_1^{(n)}(x)-g_2^{(n)}(x)){\rm
d}x,
 \]
which gives
\[
 \Bigl|\int_0^1 e^{\pm 2{\rm i}\kappa x}(g_1(x)-g_2(x)){\rm d}x\Bigr|^2\leq
(2\kappa)^{-2n}\|g^{(n)}_1-g_2^{(n)}\|^2_{H^n(0, 1)}.
\]
Therefore, we get
\begin{align*}
 \int_s^\infty \Bigl|\int_0^1 e^{\pm 2{\rm i}\kappa x}(g_1(x)-g_2(x)){\rm
d}x\Bigr|^2 {\rm d}\kappa&\leq \|g^{(n)}_1-g_2^{(n)}\|^2_{H^n(0, 1)}
\int_s^\infty (2\kappa)^{-2n}{\rm d}\kappa\\
&=\frac{s^{-(2n-1)}}{(2n-1)4^n}\|g^{(n)}_1-g_2^{(n)}\|^2_{H^n(0, 1)},
\end{align*}
which completes the proof.
\end{proof}

The following lemma is proved in \cite{CIL}.

\begin{lemm}
 Denote $S=\{z=x+{\rm i}y\in\mathbb{C}: -\frac{\pi}{4}<{\rm arg}
z<\frac{\pi}{4}\}$. Let $J(z)$ be analytic in $S$ and continuous in $\bar{S}$
satisfying
\[
 \begin{cases}
  |J(z)|\leq\epsilon, & z\in (0, ~ L],\\
  |J(z)|\leq V, & z\in S,\\
  |J(0)|=0.
 \end{cases}
\]
Then there exits a function $\mu(z)$ satisfying
\[
 \begin{cases}
  \mu(z)\geq\frac{1}{2},  & z\in(L, ~ 2^{\frac{1}{4}}L),\\
  \mu(z)\geq \frac{1}{\pi}((\frac{z}{L})^4-1)^{-\frac{1}{2}}, & z\in
(2^{\frac{1}{4}}L, ~ \infty)
 \end{cases}
\]
such that
\[
|J(z)|\leq V\epsilon^{\mu(z)}, \quad\forall\, z\in (L, ~ \infty).
\]
\end{lemm}

\begin{lemm}
 Let $f_j, g_j\in\mathcal{C}_M$. Then there exists a function
$\mu(z)$ satisfying
\begin{equation}\label{mu}
 \begin{cases}
  \mu(s)\geq\frac{1}{2}, \quad & s\in(K, ~ 2^{\frac{1}{4}}K),\\
  \mu(s)\geq \frac{1}{\pi}((\frac{s}{K})^4-1)^{-\frac{1}{2}}, \quad & s\in
(2^{\frac{1}{4}}K, ~\infty),
 \end{cases}
\end{equation}
such that
\[
 |I_1(s)|\leq CM^2 e^{3s}\epsilon_1^{2\mu(s)},\quad |I_2(s)|\leq CM^2
e^{5s}\epsilon_2^{2\mu(s)},\quad\forall s\in (K, ~\infty).
\]
\end{lemm}

\begin{proof}
We only show the proof of the estimate for $I_1(s)$ since the proof is the same
for $I_2(s)$. It follows from Lemma 3.3 that
\[
 |I_1(s)e^{-3s}|\leq CM^2,\quad\forall s\in S.
\]
Recalling \eqref{e1}, \eqref{i1}, and Lemma 3.2, we have
\[
 |I_1(s)e^{-3s}|\leq\epsilon_1^2,\quad s\in [0, ~K].
\]
A direct application of Lemma 3.5 shows that there exists a function $\mu(s)$
satisfying \eqref{mu} such that
\[
 |I_1(s)e^{-3s}|\leq CM^2\epsilon_1^{2\mu},\quad\forall s\in (K, ~\infty),
\]
which completes the proof.
\end{proof}

Now we show the proof of Theorem 2.1.

\begin{proof}
It suffices to show the estimate \eqref{cfe} since the proof is similar for
the estimate \eqref{cge}. We can assume that $\epsilon_1<e^{-1}$, otherwise the estimate is
obvious. Let
\[
s=\begin{cases}
  \frac{1}{(3\pi)^{\frac{1}{3}}}K^{\frac{2}{3}}|\ln\epsilon_1|^{\frac{1}{4}}, &
2^{\frac{1}{4}} (3\pi)^{\frac{1}{3}}K^{\frac{1}{3}}<|\ln\epsilon_1|^{\frac{1}{4}},\\
K, &|\ln\epsilon_1|\leq 2^{\frac{1}{4}}(3\pi)^{\frac{1}{3}}K^{\frac{1}{3}}.
 \end{cases}
\]
If $2^{\frac{1}{4}}(3\pi)^{\frac{1}{3}}K^{\frac{1}{3}}<|\ln\epsilon_1|^{\frac{1}{4}}$, then we have
\begin{align*}
 |I_1(s)|&\leq CM^2 e^{3s}
e^{-\frac{2|\ln\epsilon_1|}{\pi}((\frac{s}{K})^4-1)^{-\frac{1}{2}}}\leq CM^2
e^{\frac{3}{(3\pi)^{\frac{1}{3}}}K^{\frac{2}{3}}|\ln\epsilon_1|^{\frac{1}{4}}-\frac{2|\ln\epsilon_1|}{\pi}
(\frac{K}{s})^2}\\
&=CM^2
e^{-2\left(\frac{9}{\pi}\right)^{\frac{1}{3}}K^{\frac{2}{3}}|\ln\epsilon_1|^{\frac{1}{2}}\left(1-\frac{1}{2}
|\ln\epsilon_1|^{-\frac{1}{4}}\right)}.
\end{align*}
Noting $\frac{1}{2} |\ln\epsilon_1|^{-\frac{1}{4}}<\frac{1}{2}$, $\left(\frac{9}{\pi}\right)^{\frac{1}{3}}>1$  we have
\begin{align*}
 |I_1(s)|&
\leq CM^2
e^{-K^{\frac{2}{3}}|\ln\epsilon_1|^{\frac{1}{2}}}.
\end{align*}
Using
the elementary inequality
\[
 e^{-x}\leq \frac{(6n-3)!}{x^{3(2n-1)}}, \quad x>0,
\]
we get
\[
 |I_1(s)|\leq\frac{CM^2}{\left(\frac{K^2|\ln\epsilon_1|^{\frac{3}{2}}}{(6n-3)^3}\right)^{2n-1}}.
\]
If $|\ln\epsilon_1|\leq 2^{\frac{1}{4}}(3\pi)^{\frac{1}{3}}K^{\frac{1}{3}}$, then $s=K$. We have
from \eqref{e1}, \eqref{i1}, and Lemma 3.2 that
\[
 |I_1(s)|\leq \epsilon_1^2.
\]
Hence we obtain from Lemma 3.4 that
\begin{align*}
 &4\int_0^\infty \kappa^2 \left(|v(0, \kappa)|^2+|v(1, \kappa)|^2\right) {\rm
d}\kappa\\
&\leq
I_1(s)+\frac{CM^2}{\left(\frac{K^2|\ln\epsilon_1|^{\frac{3}{2}}}{(6n-3)^3}\right)^{2n-1}}+\frac{
\|f_1-f_2\|^2_{H^n(0,
1)}}{\left(2^{-\frac{1}{4}}(3\pi)^{-\frac{1}{3}}K^{\frac{2}{3}}|\ln\epsilon_1|^{\frac{1}{4}}\right)^{2n-1}}.
\end{align*}
By Lemma 3.1, we have
\[
 \|f_1-f_2\|^2_{L^2(0, 1)}\leq C \left(\epsilon_1^2
+\frac{M^2}{\left(\frac{K^2|\ln\epsilon_1|^{\frac{3}{2}}}{(6n-3)^3}\right)^{2n-1}}+\frac{M^2}{\left(\frac{K^{\frac{2}
{3}}|\ln\epsilon_1|^{\frac{1}{4}}}{(6n-3)^3}\right)^{2n-1}}\right).
\]
Since $K^{\frac{2}{3}}|\ln\epsilon_1|^{\frac{1}{4}}\leq K^2
|\ln\epsilon_1|^{\frac{3}{2}}$ when $K>1$ and $|\ln\epsilon_1|>1$, we obtain
the stability estimate.
\end{proof}

\section{Proof of Theorem \ref{mr2}}

\begin{lemm}
 Let $f_j\in L^2(0, 1), j=1, 2$ be real functions. We have for all
$\kappa\in (0,\,\infty)$ that
 \[
  2\kappa{\bf E}{\rm Re}v(0, \kappa)=\int_0^1 \sin(\kappa x)(f_1(x)-f_2(x)){\rm
d}x.
 \]
\end{lemm}

\begin{proof}
 It can be easily obtained from \eqref{sol} that
 \[
  2{\rm i}\kappa u_j(0, \kappa)=\int_0^1 e^{{\rm i}\kappa x} f_j(x){\rm
d}x+\int_0^1 e^{{\rm i}\kappa x}\sigma_j(x){\rm d}W_x.
 \]
Taking the expectation of the above equation gives
\[
 2{\rm i}\kappa {\bf E}u_j(0, \kappa)=\int_0^1 e^{{\rm i}\kappa x}f_j(x){\rm
d}x.
\]
Noting that $f_j, j=1, 2$ are real functions, we have
\[
 2{\kappa}{\bf E}{\rm Re}u_j(0, \kappa)=\int_0^1 \sin(\kappa x)f_j(x){\rm d}x,
\]
which completes the proof.
\end{proof}

\begin{lemm}
 Let $\|f_1-f_2\|_{L^2(0, 1)}\leq M$. There exists a function $\mu(\kappa)$
satisfying
 \begin{equation}\label{c1}
\begin{cases}
 \mu(\kappa)\geq\frac{1}{2}, &\kappa\in (1, ~ 2^{\frac{1}{4}}),\\
\mu(\kappa)\geq\frac{1}{\pi}(\kappa^4-1)^{-\frac{1}{2}}, &\kappa\in
(2^{\frac{1}{4}}, ~\infty),
 \end{cases}
\end{equation}
such that
\[
 \Bigl|\int_0^1 \sin(\kappa x)(f_1(x)-f_2(x)){\rm d}x \Bigr|^2\leq C M^2
e^{4\kappa}\epsilon_4^{2\mu(\kappa)},\quad\forall\kappa\in (1, ~\infty).
\]
\end{lemm}

\begin{proof}
 Let $\kappa=\kappa_1+{\rm i}\kappa_2, \kappa_1, \kappa_2\in\mathbb{R}$. It is
easy to show that
\[
 \Bigl|\int_0^1 \sin(\kappa x)(f_1(x)-f_2(x)){\rm d}x \Bigr|^2\leq
e^{|\kappa_2|}\|f_1-f_2\|_{L^2(0, 1)}.
\]
Noting $|\kappa_2|\leq\kappa_1$ for $\kappa\in S$, we have
\begin{align*}
 \bigl|e^{-2\kappa}\bigr| &\Bigl|\int_0^1 \sin(\kappa x)(f_1(x)-f_2(x)){\rm d}x
\Bigr|^2 \leq \bigl|e^{-2\kappa}\bigr| e^{|\kappa_2|} \|f_1-f_2\|_{L^2(0, 1)}\\
&\leq e^{-\kappa_1}\|f_1-f_2\|_{L^2(0, 1)}\leq M.
\end{align*}
It follows from Lemma 4.1 that
\[
 \bigl|e^{-2\kappa}\bigr| \left|\int_0^1 \sin(\kappa x)(f_1(x)-f_2(x)){\rm
d}x\right|\leq\epsilon_4,\quad \kappa\in (0, 1].
\]
We conclude from Lemma 3.5 that there exists a function $\mu$ satisfying
\eqref{c1} such that
\[
 \bigl|e^{-2\kappa}\bigr| \Bigl|\int_0^1 \sin(\kappa x)(f_1(x)-f_2(x)){\rm
d}x\leq C M \epsilon_4^{\mu(\kappa)},\quad\kappa\in (1,~\infty),
\]
which completes the proof.
\end{proof}

\begin{lemm}
 Let $f_j\in H^n(0, 1), j=1, 2$. It holds that
 \[
  \sum_{j=T}^\infty (2j\pi)^2 |{\bf E}{\rm Re}v(0, j\pi)|^2\leq
\frac{1}{T^{2n-1}}\|f_1-f_2\|^2_{H^n(0, 1)}.
 \]
\end{lemm}

\begin{proof}
 It follows from Lemma 4.1 that
 \[
  \sum_{j=T}^\infty 4(j\pi)^2 |{\bf E}{\rm Re}v(0, j\pi)|^2=\sum_{j=T}^\infty
\Bigl|\int_0^1 \sin(j\pi x)(f_1(x)-f_2(x)){\rm
d}x\Bigr|^2.
 \]
Noting that $f_j$ has a compact support in $(0, 1)$, we have from the
integration by parts that
\begin{align*}
 \Bigl|\int_0^1 \sin(j\pi x)(f_1(x)-f_2(x)){\rm
d}x\Bigr|^2&=\Bigl|\frac{1}{(j\pi)^n}\int_0^1 \sin(j\pi
x+n\pi/2)(f^{(n)}_1(x)-f^{(n)}_2(x)){\rm d}x\Bigr|^2\\
&\leq \frac{1}{(j\pi)^{2n}}\|f_1-f_2\|^2_{H^n(0, 1)}.
\end{align*}
Combining the above estimates yields
\begin{align*}
  \sum_{j=T}^\infty 4(j\pi)^2 |{\bf E}{\rm Re}v(0, j\pi)|^2 &\leq
 \left(\sum_{j=T}^\infty \frac{1}{(j\pi)^{2n}}\right)\|f_1-f_2\|^2_{H^n(0, 1)}\\
&\leq \frac{1}{\pi^{2n}}\left(\int_0^\infty \frac{1}{(T+x)^{2n}}{\rm d}x
\right)\|f_1-f_2\|^2_{H^n(0, 1)}\\
&=\frac{1}{(2n-1)\pi^{2n}}\frac{1}{T^{2n-1}}\|f_1-f_2\|^2_{H^n(0, 1)},
\end{align*}
which completes the proof.
\end{proof}

Now we show the proof of Theorem \ref{mr2}.

\begin{proof}
 We can assume that $\epsilon_3<1$, otherwise the estimate is obvious. Applying
Lemma 4.1 and the Parseval identity, we have
\begin{align*}
& \int_0^1 |f_1(x)-f_2(x)|^2{\rm d}x=\sum_{j=1}^\infty 4(j\pi)^2|{\bf E}{\rm
Re}v(0, j\pi)|^2\\
&=\sum_{j=1}^T 4(j\pi)^2|{\bf E}{\rm Re}v(0, j\pi)|^2
+\sum_{j=T+1}^\infty 4(j\pi)^2|{\bf E}{\rm Re}v(0, j\pi)|^2
\end{align*}
Let
\[
T=
 \begin{cases}
  [N^{\frac{3}{4}}|\ln\epsilon_4|^{\frac{1}{9}}], &
N^{\frac{3}{8}}<\frac{1}{2^{\frac{5}{6}}
\pi^{\frac{2}{3}}}|\ln\epsilon_4|^{\frac{1}{9}},\\
N,&N^{\frac{3}{8}}\geq \frac{1}{2^{\frac{5}{6}}
\pi^{\frac{2}{3}}}|\ln\epsilon_4|^{\frac{1}{9}}
 \end{cases}.
\]
Using Lemma 4.2 leads to
\begin{align*}
& \Bigl|\int_0^1\sin(\kappa x)(f_1(x)-f_2(x)) {\rm d}x\Bigr|^2\leq C M^2
e^{4\kappa}\epsilon_4^{2\mu(\kappa)}\leq CM^2
e^{4\kappa}e^{-2\mu(\kappa)|\ln\epsilon_4|}\\
&\leq CM^2
e^{4\kappa}e^{-\frac{2}{\pi}(\kappa^4-1)^{-\frac{1}{2}}|\ln\epsilon_4|}\leq
CM^2 e^{4\kappa-\frac{2}{\pi}\kappa^{-2}|\ln\epsilon_4|}\\
&\leq
CM^2 e^{-\frac{2}{\pi}\kappa^{-2}|\ln\epsilon_4|(1-2\pi\kappa^3|\ln\epsilon_4|^{
-1})},\quad \forall\,\kappa\in (2^{\frac{1}{4}},~\infty).
\end{align*}
Hence we have
\begin{equation}\label{c2}
 \Bigl|\int_0^1\sin(\kappa x)(f_1(x)-f_2(x)) {\rm d}x\Bigr|^2\leq CM^2
e^{-\frac{2}{\pi^3}T^{-2}|\ln\epsilon_4|(1-2\pi^4 T^3|\ln\epsilon_4|^{
-1})},\quad\forall~\kappa\in (2^{\frac{1}{4}},~T\pi].
\end{equation}

If $N^{\frac{3}{8}}<\frac{1}{2^{\frac{5}{6}}
\pi^{\frac{2}{3}}}|\ln\epsilon_4|^{\frac{1}{9}}$, then $2\pi^4
T^3|\ln\epsilon_4|^{-1}<\frac{1}{2}$ and
\begin{equation}\label{c3}
 e^{-\frac{2}{\pi^3}\frac{|\ln\epsilon_4|}{T^2}}\leq
e^{-\frac{2}{\pi^3}\frac{|\ln\epsilon_4|}{N^{\frac{3}{2}}|\ln\epsilon_4|^{\frac{
2}{9}}}}\leq
e^{-\frac{2}{\pi^3}\frac{|\ln\epsilon_4|^{\frac{7}{9}}}{N^\frac{3}{2}}}\leq
e^{-\frac{2}{\pi^3}\frac{2^5\pi^4
|\ln\epsilon_4|^{\frac{1}{9}} N^{\frac{9}{4}}}{N^{\frac{3}{2}}}}= e^{-64\pi
|\ln\epsilon_4|^{\frac{1}{9}}N^{\frac{3}{4}}}.
\end{equation}
Combining \eqref{c2} and \eqref{c3}, we obtain
\begin{align*}
  \Bigl|\int_0^1\sin(\kappa x)(f_1(x)-f_2(x)) {\rm d}x\Bigr|^2&\leq CM^2
e^{-\frac{2}{\pi^3}T^{-2}|\ln\epsilon_4|(1-2\pi^4 T^3|\ln\epsilon_4|^{
-1})}\\
&\leq CM^2 e^{-\frac{1}{\pi^3}T^{-2}|\ln\epsilon_4|}\leq CM^2 e^{-32\pi
|\ln\epsilon_4|^{\frac{1}{9}} N^{\frac{3}{4}}},\quad\forall\kappa\in
(2^{\frac{1}{4}},~T\pi].
\end{align*}
It is easy to note that
\[
 e^{-x}\leq \frac{(6n-3)!}{x^{3(2n-1)}}\quad\text{for} ~ x>0.
\]
We have
\[
  \Bigl|\int_0^1\sin(j\pi x)(f_1(x)-f_2(x)) {\rm d}x\Bigr|^2\leq
CM^2\frac{1}{\left(\frac{|\ln\epsilon_4|^{\frac{1}{3}}N^{\frac{9}{4}}}{(6n-3)^3}\right)^{2n-1}},
\quad j=1, \dots, T.
\]
Consequently, we obtain
\begin{align*}
& \sum_{j=1}^T  \Bigl|\int_0^1\sin(j\pi x)(f_1(x)-f_2(x)) {\rm d}x\Bigr|^2\leq
CM^2\frac{T}{\left(\frac{|\ln\epsilon_4|^{\frac{1}{3}}N^{\frac{9}{4}}}{(6n-3)^3}\right)^{2n-1}}\\
&\leq CM^2\frac{N^{\frac{3}{4}}|\ln\epsilon_4|^{\frac{1}{9}}}{\left(\frac{
|\ln\epsilon_4|^{\frac{1}{3}}N^{\frac{9}{4}}}{(6n-3)^3}\right)^{2n-1}}\leq CM^2
\frac{1}{\left(\frac{|\ln\epsilon_4|^{\frac{2}{9}}N^{\frac{3}{2}}}{(6n-3)^3}\right)^{2n-1}}\\
&\leq CM^2
\frac{1}{\left(\frac{|\ln\epsilon_4|^{\frac{1}{9}}N^{\frac{3}{2}}}{(6n-3)^3}\right)^{2n-1}}.
\end{align*}
Here we have noted that $|\ln\epsilon_4|>1$ when $N^{\frac{3}{8}}<\frac{1}{2^{\frac{5}{6}}
\pi^{\frac{2}{3}}}|\ln\epsilon_4|^{\frac{1}{9}}$.
If $N^{\frac{3}{8}}<\frac{1}{2^{\frac{5}{6}}
\pi^{\frac{2}{3}}}|\ln\epsilon_4|^{\frac{1}{9}}$, we also have
\begin{align*}
\frac{1}{\left(\bigl[|\ln\epsilon_4|^{\frac{1}{9}}N^{\frac{3}{4}}\bigr]+1\right)^{2n-1}}
\leq
\frac{1}{\left(|\ln\epsilon_4|^{\frac{1}{9}}N^{\frac{3}{4}}\right)^{2n-1}}.
\end{align*}

If $N^{\frac{3}{8}}\geq \frac{1}{2^{\frac{5}{6}}
\pi^{\frac{2}{3}}}|\ln\epsilon_4|^{\frac{1}{9}}$, then $T=N$. It follows from
Lemma 4.1 that
\[
 \sum_{j=1}^T  \Bigl|\int_0^1\sin(j\pi x)(f_1(x)-f_2(x)) {\rm
d}x\Bigr|^2=\epsilon_3^2.
\]
Combining the above estimates, we obtain
\begin{align*}
 \Bigl|\int_0^1\sin(\kappa x)(f_1(x)-f_2(x)) {\rm d}x\Bigr|^2\leq
C\epsilon_3^2+CM^2
\frac{1}{\left(\frac{|\ln\epsilon_4|^{\frac{1}{9}}N^{\frac{3}{2}}}{(6n-3)^3}\right)^{2n-1}}\\
+CM^2\frac{1}{\bigl(|\ln\epsilon_4|^{\frac{1}{9}}N^{\frac{3}{4}}
\bigr)^{2n-1}}+CM^2
\frac{(2^{\frac{5}{6}}\pi^{\frac{2}{3}})^{2n-1}}{\left(|\ln\epsilon_4|^{\frac{1}{9}}N^{\frac{5}{8}}
\right)^{2n-1}}.
\end{align*}
Noting that $N^{\frac{5}{8}}\le N^{\frac{3}{4}} \le N^{\frac{3}{2}}$ and $2^{\frac{5}{6}}\pi^{\frac{2}{3}}\le (6n-3)^3$, $\forall n\in\mathbb{N}$.
The proof is completed by combining the above estimates.
\end{proof}

\end{document}